\documentclass[11pt,oneside]{article}
 \usepackage{afterpage}
 \usepackage{fancyhdr}
 \usepackage{lipsum}
 \newcommand\shorttitle{Curvatures on homogeneous Finsler spaces}
 \newcommand\authors{G. Shanker, Seema and Jaspreet Kaur}
 
 \usepackage{supertabular, setspace,hyperref}
 \usepackage[a4paper, total={6.5in, 9.5in}]{geometry}
 \usepackage[utf8]{inputenc}
 \usepackage{amsmath}
 \usepackage{mathtools}
 \usepackage{amsfonts}
 \usepackage{amssymb}
 \usepackage{amsthm}
 \usepackage{caption}
 \usepackage{subcaption}
 \usepackage{authblk}
 \usepackage[sort]{cite}
 \newtheorem{theorem}{Theorem}[section]

 \newtheorem{example}[theorem]{Example}

 \newtheorem{remark}{\sc Remark}
 \newtheorem{lemma}{\sc Lemma}[section]
 \newtheorem{corollary}{\sc Corollary}[section]
 \newtheorem{definition}{\sc Definition}[section]

 \newcommand{\be}{\begin{eqnarray}}
 	\newcommand{\ee}{\end{eqnarray}}
 \newcommand{\Be}{\begin{eqnarray*}}
 	\newcommand{\Ee}{\end{eqnarray*}}
 \newcommand{\bee}{\begin{equation}}
 	\newcommand{\eee}{\end{equation}}
 \newcommand{\ba}{\begin{array}}
 	\newcommand{\ea}{\end{array}}
 \newcommand{\bl}{\begin{lemma}}
 	\newcommand{\el}{\end{lemma}}
 \newcommand{\bd}{\begin{definition}}
 	\newcommand{\ed}{\end{definition}}
 \newcommand{\bt}{\begin{theorem}}
 	\newcommand{\et}{\end{theorem}}
 \newcommand{\bp}{\begin{proof}}
 	\newcommand{\ep}{\end{proof}}
 \newcommand{\bi}{\begin{itemize}}
 	\newcommand{\ei}{\end{itemize}}
 \newcommand{\br}{\begin{remark}}
 	\newcommand{\er}{\end{remark}}
 \newcommand{\bc}{\begin{corollary}}
 	\newcommand{\ec}{\end{corollary}}
 \newcommand{\bex}{\begin{example}}
 	\newcommand{\eex}{\end{example}}

 \usepackage{chngcntr}
 
 \counterwithin*{equation}{section}
 \counterwithin*{equation}{section}
 
\begin{document}
 	\afterpage{\cfoot{\thepage}}
 	\clearpage
 	\date{}

 	\title{\textbf{Curvatures on homogeneous Finsler spaces}}
 	
 	\author[]{Gauree Shanker}
 	\author[]{Seema Jangir\thanks{corresponding author, Email: seemajangir2@gmail.com}}
    \author[]{Jaspreet Kaur}
 	\affil[]{\footnotesize Department of Mathematics and Statistics,\\
 		Central University of Punjab, Bathinda, \\Punjab-151 001, India}
 	\maketitle
 	\begin{center}
 		\textbf{Abstract}
 		
 	\end{center}
 	\begin{small}
 	The main aim of this article is to calculate explicit formula for S-curvature in homogeneous generalized $m$-Kropina metric.
 	Further, we also deduce mean Berwald curvature for homogeneous generalized $m$-Kropina metric from S-curvature.\\
 	\end{small}

 	\section{Introduction}
 		
It is well-known that in Riemannian geometry, inner product spaces in any dimension upto isomorphism are unique, but this doesn't hold in the case of Minkowski norm.
Finsler geometry encompasses Riemannian, Euclidean, and Minkowskian geometries as special cases, and thus it affords great
generality for describing a number of phenomena in physics and continuum mechanics.
In continuum mechanical behaviour of solids, the importance of Finsler geometry was observed by many physicists \cite{Kroner,Clayton2015}. Along with mechanics, importance of Finsler geometry in describing fundamental descriptions of other branches of  physics such as electromagnetism, quantum theory and gravitations is striking.
As we know that various Finsler metrics on tangent space of a Finsler manifold are not isomorphic.\\
 In Shen's words, a Finsler manifold is quite colorful. Thus, it is very fascinating to study Finsler manifolds with single color. Homogeneous Finsler spaces
are examples of Finsler manifolds with single color, which have been studied by Deng \cite{Deng2012} Latifi and Razavi \cite{Latifi}.
The origin of homongenous Riemannian spaces was initiated with Myers-Steenrod theorem in [1939] which states that group of isometries of a connected Riemannian manifold admits a differentiable structure such that it forms a Lie transformation group of manifold.
This theorem was breakthrough, since it extended the scope of applying Lie theory to all homogeneous Riemannian manifolds. Homogeneous spaces are natural generalization of symmetric spaces and these spaces retain many of its great properties. One of the important properties is the existence of a transitive group of transformations, which are sometimes called symmetries.\\
Deng \cite{Deng2012} and Shen \cite{Shen} examined the non-Riemannian curvatures such as Cartan torsion, Landsberg curvature, mean Landsberg curvature, Berwald curvature, S-curvature which disappear in the case of Riemannian \cite{Tayebi}. While studying curvature properties of homogeneous spaces by using the formula for the sectional curvature of a left invariant Riemannian metric on a Lie group, Milnor \cite{Milnor} calculated S-curvature in homogeneous spaces.
Geometrically, S-curvature is study of rate of change of distortions along geodesics.
It is known that S-curvature is a non-Riemannian quantity means every Riemannian manifold has vanishing S-curvature.
The notion of  S-curvature was introduced by Shen \cite{Shen1997}
for given comparison theorems on Finsler manifolds in 1997. This non-Riemannian quantity is used for characterization of Finsler metrics among Berwald metric, Riemannian metric and Locally
Mikowskian metric. Shen \cite{Chern2004} also gave explicit formula for S-curvature in local coordinate system. 

  However, in homogeneous Finsler spaces, mathematicians were interested in finding S-curvature formula irrespective of local coordinate system. For homogeneous Randers space, S-curvature formula was calculated by Deng \cite{Deng2009}.  Further, calculation of S-curvature on  homogeneous $(\alpha,\beta)$-metrics was another task completed by Deng and Wang \cite{Deng2010}.
Recently, Shanker and kaur rectified the S-curvature formula for homogeneous $(\alpha,\beta)$-metric provided by Deng and Wang \cite{Deng2010}.
Later, calculation of S-curvature on a homogeneous Finsler space with square metric and Randers changed square metric have been accomplished in \cite{Sarita,sarita2}.
The purpose of this paper is calculating S-curvature and mean Berwald curvature in homogeneous generalized $m$-Kropina Finsler spaces.

\section{Preliminaries}

In this section, we discuss basic definitions and notations of Finsler geometry. For more elaborate concepts of Finsler geometry and homogeneous Finsler geometry, refer \cite{BCS,Chern2004,Deng2012 }. 

\begin{definition}
	
	Let $V$ be an $n$-dimensional real vector space endowed with a smooth norm $F$ on $V\backslash\{0\}$ satisfying the following conditions:
	\begin{enumerate}
		\item $F(u) \geq 0 \ \forall \ u \ \in V,$
		\item $F(\lambda u) = \lambda F(u) \ \forall \ \lambda >0,$ i.e., F is positively homogeneous,
		\item Let $\{ u_{1}, u_{2},...,u_{n}\}$ be the basis of $V$ such that $y = y^{1}u_{1} + y^{2}u_{2}+...+y^{n}u_{n}$.
		Then the Hessian matrix 
		\begin{center}
			$(g_{ij}) := \left( \left[ \dfrac{1}{2} F^{2}\right]_{y^{i}y^{j}}\right), $  
		\end{center}
		is positive definite at every point of $V\backslash \{0\}$. The pair $(V,F)$ is called Minkowski space and $F$ is called Minkowski norm.
	\end{enumerate}
\end{definition}
For example, a Minkowski norm $F$ is said to be Euclidean or coming from inner product, if $\langle, \rangle$ is an inner product on $V,$ and we define $F(y) = \sqrt{\langle y,y \rangle}$ as Minkowski norm.

\begin{definition}
	
	Let $M$ be a connected (smooth) manifold. A Finsler metric on $M$ is a function $F : TM \rightarrow [0,\infty)$ which satisfies:
	\begin{enumerate}
		\item $F$ is smooth on slit tangent bundle $TM\backslash\{0\},$
		\item The restriction of $F$ to any $T_{x}M, x\in M $ is a Minkowski norm.
		
	\end{enumerate}
	The space $(M,F)$ is called Finsler space.
	
\end{definition}

Let $\gamma : [0,1] \rightarrow M$ be a $C^{1}$-curve. Then Finsler length $L(\gamma)$ of $\gamma$ is defined as  
\begin{center}
	$ L(\gamma)= \int_{0}^{1} F(\gamma(t),\gamma^{'}(t))dt$.
	
\end{center}
Further,  Finsler distance $d_{F}(p,q)$ between two points $p,q \in M$ is defined as 
\begin{center}
	$d_{F} (p,q) =inf_{\gamma} L(\gamma),$ 
\end{center}    where infimum is taken over all piecewise $C^{1}$-curves joining $p$ and $q$.

\begin{definition}
	Let $ F = \alpha \phi(s) ; s = \beta/\alpha,$
	where $\phi$ is a smooth function on an open interval $(-b_{0},b_{0}),\ \alpha = \sqrt{a_{ij}(x)y^{i}y^{j}}$ is a Riemannian metric, $ \beta = b_{i}(x) y^{i}$  is a 1-form on an $n$-dimensional manifold with $||\beta || < b_{0}.$ Then, $F$ is Finsler metric if and only if following conditions are satisfied:
	\begin{equation}{\label{a}}
	\phi(s)>0,\  \phi(s)- s\phi^{'}(s)+ (b^{2}-s^{2}) \phi^{''}(s) >0 \ \forall \ |s| \leq b < b_{0}.
	\end{equation}
	In this, if $\phi$ does not satisfy the equation \ref{a} or $\phi(0)$ is not defined, then $(\alpha,\beta)$-metric is said to be singular Finsler metric.
\end{definition}
If we take  $\phi(s) = \dfrac{1}{s^{m}}, (s \neq 0)$, we get an important class of $(\alpha,\beta)$- metrics known as generalized $m$-Kropina metric. Hence, generalized $m$-Kropina metric is defined as $ F = \dfrac{\alpha^{m+1}}{\beta^{m}}.$ \\
In another way, generalized $m$-Kropina metric can also be expressed using the following Lemma:

\begin{lemma}{\label{Lemma3.1}}
	
	Let $(M,\alpha)$ be a Riemannian space. Then the generalized m-Kropina space, $(M,F)$ where $F = \dfrac{\alpha^{m+1}}{\beta^{m}}, (m \neq -1,0,1)\ \beta =  b_{i}y^{i},$ a 1-form with $||\beta||= \sqrt{b_{i} b^{i}} < 1,$ consists of Riemannian metric $\alpha$ along with a smooth vector field $X$ on $M,$ which satisfies $\alpha(X|_{x}) < 1$ $\forall \ x\in M,$ i.e., 
	$$ F(x,y) = \dfrac{\alpha(x,y)^{m+1}}{\langle X|_{x},y \rangle ^{m}}.$$
	
\end{lemma}

\begin{definition}
Consider an $n$-dimensional vector space $V,$ with $F$ be Minkowski norm on $V$. For $\{e_{i}\}$ to be the basis of $V,$ define a quantity as
$$ \sigma_{F} = \dfrac{\text{Vol}(B^{n})}{\text{Vol}\{ (y^{i}) \in \mathbb{R}^{n} | F(y^{i}e^{i})< 1\}},$$
 where Vol denotes Volume of a subset in Euclidean space $\mathbb{R}^{n}$ and $B^{n}$ is the open ball with radius of 1. Generally, this quantity is dependent on basis $\{e_{i}\}$.
\end{definition}
Next, we recall the definition of S-curvature in Finsler spaces:
\begin{definition}
	The distortion of an $n$-dimensional vector space is defined as 
	\begin{center}
		$ \tau(y) = ln \dfrac{\sqrt{det(g_{ij}(y))}}{\sigma_{F}}$.
	\end{center} For $x\in M,$ let $\tau(x,y)$ be distortion of Minkowski norm $F_{x}$ on $T_{x}M, \ x\in M$. Let $\gamma(t)$ be a geodesic satisfying $\gamma(0) = x \  \& \ \dot\gamma(0) = y. $ Then rate of change of distortion along geodesic $\gamma$ is said to be S-curvature and is defined as 
	\begin{center}
		
		$S(x,y) = \dfrac{d}{dt}[ \tau(\gamma(t)),\dot\gamma(t)]\Big|_{t=0}$.
		
	\end{center} 
\end{definition}

\begin{definition}{\label{iso}}
	A Finsler space $(G/H, F)$ is said to have almost isotropic S-curvature, if there exists a smooth function $d(x)$ on $M$ and closed 1-form $\eta$ which satisfies:
	$$ S(x,y) = (n+1)(d(x)F(y)+\eta(y)),\ x\in G/H, \ y\in T_{x}M.$$
	If $\eta(y) = 0$, then $(G/H,F)$ is said to have isotropic S-curvature.\\
	And if $\eta(y) = 0, c(x) =$ constant, then $(G/H,F)$ is said to have constant S-curvature.
\end{definition}
\begin{definition}{\label{BHT}}
	The Busemann-Hausdroff volume form, $dV_{BH} = \sigma_{BH}(x)dx$ is defined as:
	$$ \sigma_{BH} =  \dfrac{\text{Vol}(B^{n})}{\text{Vol}\{ (y^{i}) \in \mathbb{R}^{n} | F(x,y^{i}e^{i})< 1\}}.$$
	The Holmes-Thompson volume form, $dV_{HT} = \sigma_{HT}(x)dx$ is defined as :
	$$ \sigma_{HT}(x) = \dfrac{1}{\text{Vol}B^{n}} \int_{\{y^{i}\in \mathbb{R}^{n}| F(x,y^{i}e^{i})<1\}} det(g_{ij})dy.$$
	Both volume forms coincide in the case of Riemannian metric, i.e., $dV_{BH} = dV_{HT}=\sqrt{det g_{ij}(x)dx.}$
	  
\end{definition}

\begin{definition}
	Let $(M,F)$ be a Finsler space. A diffeomorphism of $M$ onto itself is said to be isometry, if it preserves the Finsler function, i.e., $F(\phi(p), d\phi_{p}(X)) = F(p,X)$ for any $p\in M$ and $X\in T_{p}M$.
\end{definition}

\begin{definition}
	Let $G$ be a Lie group and $M$ a smooth manifold. If $G$ has smooth action on $M$, then $G$ is called Lie transformation group of $M$.
	
\end{definition}

\begin{definition}
	A connected Finsler space $(M,F)$ is said to be homogeneous Finsler space, if action of group of isometries of $(M,F)$, denoted by $I(M,F)$ is transitive on $M$. 
\end{definition}

\section{S-curvature homogeneous generalized $m$-Kropina space }
Shen and Cheng \cite{ChengShen} in 2009, discovered the formula for S-curvature of an $(\alpha,\beta)$-metric in local coordinate system, which is given as follows:
\begin{equation}{\label{local}}
 S = \left( 2\psi - \dfrac{f^{'}(b)}{bf(b)} \right) (r_{0} + s_{0})- \alpha^{-1} \dfrac{\Phi}{2 \Delta^{2}} (r_{00} - 2\alpha Q s_{0}), \ \text{where}
 \end{equation} 
$$Q = \dfrac{\phi^{'}}{\phi-s\phi^{'}}, \ \psi = \dfrac{Q^{'}}{2\Delta},$$
$$\Delta = 1+sQ+(b^{2}-s^{2})Q^{'}, $$ 
$$\Phi = (sQ^{'}-Q)(1+n\Delta+sQ)-(b^{2}-s^{2})(1+sQ)Q^{''},$$ 
$$ r_{ij} := \dfrac{1}{2} (b_{i|j}+b_{j|i}), \ r_j = b^{i}r_{ij}, \ r_{0} = r_{i}y^{i}, \  r_{00}= r_{ij}y^{i}y^{j},$$ 
$$s_{ij} := \dfrac{1}{2} (b_{i|j}-b_{j|i}), \ s_{j} = b^{i}s_{ij}, \ s_{0} = s_{i}y^{i},$$

And $f(b)$ used in equation \ref{local} is defined as\\ \\
$$
f(b) =  \begin{cases}
\dfrac{\int_{0}^{\pi}sin^{n-2}tdt}{\int_{0}^{\pi}\dfrac{sin^{n-2}t}{\phi(bcost)^{n}}dt} &\quad\text{if}\ dV = dV_{BH}, \\ \\
\dfrac{\int_{0}^{\pi} (sin^{n-2}t)T(bcost)dt}{\int_{0}^{\pi}(sin^{n-2}t)dt } &\quad\text{if} \ dV = dV_{HT}. \\
\end{cases}$$\\ \\ 
where, $dV_{BH}, dV_{HT}$ are Busemann-Hausdroff volume form and Holmes-Thompson volume form  are defined in defintion \ref{BHT}.\\
 We can easily see that in the case of constant Riemannian length $b$, the parameter $r_{0}+s_{0}$ vanishes, hence equation \ref{local} reduces to 
 $$S = \alpha^{-1} \dfrac{\Phi}{2 \Delta^{2}} (r_{00} - 2\alpha Q s_{0})$$

Next, we recall the amendment done in S-curvature formula in homogeneous Finsler spaces proved in \cite{Kiran}.

\begin{theorem} {\cite{Kiran}} {\label{Thm1}}
Let $F = \alpha \phi(s)$ be a $G$-invariant $(\alpha, \beta)$- metric on the reductive homogeneous Finsler space $G/H$ with a decomposition of the Lie algebra $\mathfrak{g} = \mathfrak{h} +\mathfrak{m}.$ Then S-curvature is given by 
\begin{equation}{\label{eq1}}
S(H,y) =  \dfrac{\Phi}{2\alpha\Delta^{2}} \Bigg( \Big \langle [v,y]_{\mathfrak{m}}, y \Big \rangle + \alpha Q \Big \langle [v,y]_{\mathfrak{m}},v \Big \rangle  \Bigg),
\end{equation}
where, $v\in \mathfrak{m}$ corresponds to the 1-form $\beta$ and $\mathfrak{m}$ is identified with the tangent space $T_{H}(G/H)$ of $G/H$ at the origin $H.$
\end{theorem}
Using Theorem \ref{Thm1}, we deduce formula for S-curvature in the homogeneous generalized $m$-Kropina spaces.

\begin{theorem}{\label{thm2}}
	Let $G/H$ be a reductive homogeneous Finsler spaces with decomposition of Lie algebra $\mathfrak{g} = \mathfrak{h} +\mathfrak{m},$ and $F = \dfrac{\alpha^{m+1}}{\beta^{m}}$ be a $G$-invariant generalized $m$-Kropina metric. Then S-curvature is given by 
	\begin{equation}{\label{eq2}}
\dfrac{ms[(n-nm)s^{2}+(nm+1)b^{2}]}{[(1-m)s^{2}+b^{2}m]^{2}}\Big[ \dfrac{1}{\alpha} \big \langle [v,y]_{\mathfrak{m}},y\big \rangle-\dfrac{m}{(m+1)s} \big \langle [v,y]_{\mathfrak{m}},v   \big \rangle  \Big], 
	\end{equation}
where, $v\in \mathfrak{m}$ corresponds to the 1-form $\beta$ and $\mathfrak{m}$ is identified with the tangent space $T_{H}(G/H)$ of $G/H$ at the origin $H.$
\end{theorem}
\begin{proof}
	In order to calculate S-curvature in reductive homogeneous Finsler space with $G$-invariant generalized $m$-Kropina metric, we first find the parameters to be used in modified version of S-curvature in Theorem \ref{Thm1}.\\
	Let us recall that\\
	$Q = \dfrac{\phi^{'}}{\phi-s\phi^{'}},$\\ \\
	$\Delta = 1+sQ+(b^{2}-s^{2})Q^{'},$\\ \\
	$\Phi = -(Q-sQ^{'})(1+n\Delta+sQ)-(b^{2}-s^{2})(1+sQ)Q^{''},$\\ \\
	For  $F = \dfrac{\alpha^{m+1}}{\beta^{m}},$ we have $F = \alpha\left( \dfrac{1}{s^{m}}\right),$ such that  $\phi(s)= \dfrac{1}{s^{m}}, s=\dfrac{\beta}{\alpha}.$\\
Now we proceed to calculate these paramters for generalized $m$-Kropina metric.\\ \\
$Q = \dfrac{\phi^{'}}{\phi-s\phi^{'}}
= \dfrac{-m}{s(1+m)},$\\ \\
$Q^{'} = \left( \dfrac{m}{1+m}\right) \dfrac{1}{s^{2}},$\\ \\
$Q^{''} = \left( \dfrac{-2m}{1+m}\right) \dfrac{1}{s^{3}},$\\ \\
$\Delta = 1+sQ+(b^{2}-s^{2})Q^{'}
= 1 + \dfrac{(-ms)}{(m+1)s}+(b^{2}-s^{2})\left( \dfrac{m}{(1+m)s^{2}}\right) 
= \dfrac{(1-m)s^{2}+b^{2}m}{(1+m)s^{2}},$\\ \\
$\Phi = (sQ^{'}-Q)(1+n\Delta+sQ)-(b^{2}-s^{2})(1+sQ)Q^{''} \\ \\
= \Big(\dfrac{2m}{m+1}\Big)\dfrac{1}{s}\Bigg[ \dfrac{ns^{2}+nb^{2}m-nms^{2}}{(1+m)s^{2}}  +1- \dfrac{ms}{(m+1)s}\Bigg] + \dfrac{2m}{(1+m)^{2}}\left( \dfrac{b^{2}-s^{2}}{s^{3}} \right) $\\ \\
$= \dfrac{2m}{(1+m)^{2}s^{3}}$$\Big[ns^{2} + nb^{2}m - nms^{2} +s^{2} +b^{2}-s^{2}  \Big]
 =  \dfrac{2m}{(1+m)^{2}s^{3}}\Big[ (n-nm)s^{2}+(nm+1)b^{2}\Big].$ \\ \\
Now, we substitute all these values in equation \ref{eq1}, which gives \\ 
	$$S(H,y) = \dfrac{ms[(n-nm)s^{2}+(nm+1)b^{2}]}{[(1-m)s^{2}+b^{2}m]^{2}}\Bigg[ \dfrac{1}{\alpha} \big \langle [v,y]_{\mathfrak{m}},y\big \rangle-\dfrac{m}{(m+1)s} \big \langle [v,y]_{\mathfrak{m}},v   \big \rangle  \Bigg] $$
	where, $v\in \mathfrak{m}$ corresponds to the 1-form $\beta$ and $\mathfrak{m}$ is identified with the tangent space $T_{H}(G/H)$ of $G/H$ at the origin $H.$
	This proves Theorem \ref{thm2}..
\end{proof}
A straightforward application of Theorem \ref{thm2} can be seen as:
\begin{corollary}
	Let $(G/H, F)$ be a homogeneous generalized $m$-Kropina Finlser space with similar presumptions as taken in Theorem \ref{thm2}. Then homogeneous generalized $m$-Kropina metric has isotropic S-curvature if and only if S-curvature of $F$ equal to zero. 
	
\end{corollary}
\begin{proof}
	Using the definition \ref{iso}, we recall that if $(G/H,F)$ has isotropic S-curvature implies 
	$$ S(x,y) = (n+1)(d(x)F(y)+\eta(y)),\ x\in G/H, \ y\in T_{x}M.$$
	Since in the case of homogeneous space, it is enough to calculate S-curvature at origin, i.e., $x = H, \ y=v$. Put $x= H, \ y=v$ in the formula of S- curvature deduced in Theorem \ref{thm2}, we get $d(H)=0,$ which implies $S(H,y)=0\ \forall y\in T_{x}M$. Hence, $G/H$ has zero S-curvature.

\end{proof}
\section{Mean Berwald curvature}
This section emphasises on calculating mean Berwald curvature for homogeneous Finsler spaces with generalized $m$-Kropina metric.\\
We first recall the concept of mean Berwald curvature in homogeneous Finsler spaces. In \cite{Chern2004}, author discussed the notion of mean Berwald curvature  for Finsler spaces as follows:\\

Consider $E_{ij} = \dfrac{1}{2}\dfrac{\partial^{2}S(x,y)}{\partial y^{i} \partial y^{j}} = \dfrac{1}{2} \dfrac{\partial^{2}}{\partial y^{i} \partial y^{j}}\left( \dfrac{\partial G^{m}}{\partial y^{m}} \right)(x,y),$ where $G^{m}$ are geodesic spray coefficients.\\ \\
Consider the symmetric forms $E_{y}:T_{x}M \times T_{x}M \rightarrow \mathbb{R}$ expressed as $$E_{y}(a,b) = E_{ij}a^{i}b^{j},\text{where} \ a = a^{i} \dfrac{\partial}{\partial y^{i}}, b = b^{j}\dfrac{\partial}{\partial y^{j}}\in T_{x}M, \ x\in M. $$
Then E-tensor defined as family of symmetric forms $\{E_{y}\}$ can be seen as:  $$E = \Big \{  E_{y}: y \in TM/\{0\} \Big \}, $$ which is known as E-Curvature or mean Berwald curvature.\\
Before moving to main result of this section let us find some important quantities to be used further:
We know that at origin,\\ 
 $ a_{ij} = \delta^{j}_{i}, \ y_{i} = y^{i}$,\\ \
$\alpha_{y_{i}}= \dfrac{y_{i}}{\alpha}, \ \beta_{y^{i}} = b_{i},$\\
$s_{y^{i}} = \dfrac{\partial}{\partial y^{i}}\left( \dfrac{\beta}{\alpha}\right)  = \dfrac{\alpha b_{i} - sy_{i}}{\alpha^{2}},$\\ \\
$ s_{y_{i}y_{j}} = \dfrac{\partial}{\partial y^{j}}\left(\dfrac{\alpha b_{i} - sy_{i}}{\alpha^{2}} \right) \\ \\
=\dfrac{3sy_{i}y_{j} - (b_{i}y_{j}+b_{j}y_{i})\alpha - \alpha^{2}s\delta^{i}_{j}}{\alpha^{4}}.$\\ \\

Consider $\Omega  = \dfrac{ms[(n-nm)s^{2} + (nm+1)b^{2}]}{[(1-m)s^{2}+b^{2}m]^{2}} $
= $\dfrac{mn(1-m)s^{3}+m(nm+1)b^{2}s}{[(1-m)s^{2}+b^{2}m]^{2}}$

\begin{equation}
\begin{split}
\dfrac{\partial \Omega}{\partial y^{j}} = & \dfrac{1}{{[(1-m)s^{2}+b^{2}m]^{4}}}\Bigg[  [(1-m)s^{2}+b^{2}m]^{2}(3mn(1-m)s^{2}s_{y^{j}}+m(nm+1)b^{2}s_{y^{j}}) \\
& -[mn(1-m)s^{3}+m(nm+1)b^{2}s][2(1-m)s^{2}+b^{2}m](1-m)2ss_{y^{j}} \Bigg]   \\ 
= \     &\dfrac{1}{[(1-m)s^{2}+b^{2}m]^{3}}\Bigg[3mn(1-m)^{2}s^{4}+m(1-m)(1+nm+3mnb^{2})s^{2}\\& -4mn(1-m)s^{3}+m(nm+1)b^{2}s +m^{2}(nm+1)b^{4}\Bigg] s_{y^{j}},
\end{split}
\end{equation}

\begin{equation}
\begin{split}
\dfrac{\partial^{2} \Omega }{\partial y^{i}\partial y^{j}} = & \dfrac{\partial}{\partial y^{i}}\left[\dfrac{\splitfrac{ \{3mn(1-m)^{2}s^{4}+m(1-m)(1+nm+3mnb^{2})s^{2}}{-4mn(1-m)s^{3}+m(nm+1)b^{2}s +m^{2}(nm+1)b^{4}\} s_{y^{j}}} }{[(1-m)s^{2}+b^{2}m]^{3}}   \right] \\
= & \dfrac{\partial}{\partial y^{i}}\left[\dfrac{\splitfrac{ \{3mn(1-m)^{2}s^{4}+m(1-m)(1+nm+3mnb^{2})s^{2}}{-4mn(1-m)s^{3}+m(nm+1)b^{2}s +m^{2}(nm+1)b^{4}\} } }{[(1-m)s^{2}+b^{2}m]^{3}}   \right] s_{y^{j}}  \\
+ & \left[ \dfrac{\splitfrac{\{3mn(1-m)^{2}s^{4}+m(1-m)(1+nm+3mnb^{2})s^{2}}{-4mn(1-m)s^{3}+m(nm+1)b^{2}s +m^{2}(nm+1)b^{4}\} s_{y^{j}}} }{[(1-m)s^{2}+b^{2}m]^{3}}   \right] s_{y^{j}{y^{i}} }\\
\end{split}
\end{equation}


\begin{equation}
	\begin{split}
	\dfrac{\partial^{2} \Omega }{\partial y^{i}\partial y^{j}} = &
	\left[\dfrac{\splitfrac{-6mn(1-m)^{3}s^{5}+12mn(1-m)^{2}s^{4}-4m(1-m)^{2}(1+nm)s^{3}}{-m(1-m)(5+17mn)b^{2}s^{2} +2m^{2}(1-m)(mn-2)b^{2}s +m^{2}(1+mn)b^{4}}}{[(1-m)s^{2}+b^{2}m]^{4}}  \right] s_{y^{j}} s_{y^{i}} \\
	+ & \left[ \dfrac{\splitfrac{\{3mn(1-m)^{2}s^{4}+m(1-m)(1+nm+3mnb^{2})s^{2}}{-4mn(1-m)s^{3}+m(nm+1)b^{2}s +m^{2}(nm+1)b^{4}\} s_{y^{j}}} }{[(1-m)s^{2}+b^{2}m]^{3}}   \right] s_{y^{i}{y^{j}} }.\\ 
	\end{split}
\end{equation}

\begin{theorem}
	Let $(G/H,F)$ be a reductive homogeneous Finsler spaces with decomposition of Lie algebra $\mathfrak{g} = \mathfrak{h} +\mathfrak{m},$ and $F = \dfrac{\alpha^{m+1}}{\beta^{m}}$ be a $G$-invariant generalized $m$- Kropina metric. Then mean Berwald curvature is given by 
	\begin{equation}
	\begin{split}
	E_{i,j}(H,y) = & \dfrac{1}{2}\Bigg[ \Big(\frac{1}{\alpha}\Omega_{ij}- \dfrac{y_{i}}{\alpha^{3}} \Omega_{j} -\dfrac{y^{j}}{\alpha^{3}}\Omega_{i}-\dfrac{A}{\alpha^{3}}\delta^{j}_{i}+\dfrac{3\Omega}{\alpha^{5}}y^{i}y^{j}  \Big) \big \langle [v,y]_{m},y \big \rangle \\
	& + \Big( \dfrac{1}{\alpha}\Omega_{j} - \dfrac{\Omega y_{j}}{\alpha^{3}}  \Big) \Big( \big \langle [v,v_{i}]_{m} ,y \big \rangle + \big \langle [v,y]_{m}, v_{i} \big \rangle \Big) \\
	& +\Big( \dfrac{1}{\alpha}\Omega_{i} - \dfrac{\Omega y_{i}}{\alpha^{3}}  \Big) \Big(  \big \langle [v,v_{j}]_{m} ,y \big \rangle + \big \langle [v,y]_{m}, v_{j} \big \rangle  \Big) \\
	& + \dfrac{\Omega}{\alpha}\Big( \  \big \langle [v,v_{j}]_{m} ,v{i} \big \rangle + \big \langle [v,v_{i}]_{m}, v_{j} \big \rangle   \Big) 
	 + \Big( \dfrac{-m}{m+1} \dfrac{\Omega_{i}}{s^{2}}s_{y^{j}} 
	  + \dfrac{2m}{(m+1)s^{3}} \Omega s_{y^{i}} s_{y^{j}} \\
	  & - \dfrac{m}{(m+1)s^{2}} \Omega s_{y^{j}y^{i}} + \dfrac{m}{(m+1)s} \Omega_{ji} - \dfrac{m}{(m+1)s^{2}}s^{y^{i}}\Omega_{j} \Big)   \big \langle [v,y]_{m},v  \big \rangle  \\
	  & + \Big( \dfrac{-m\Omega}{(m+1)s^{2}}s_{y^{j}}+ \dfrac{m}{(m+1)s} \Omega_{j} \Big) \big \langle [v,v_{i}]_{m},v  \big \rangle 
	 	\Bigg]
\end{split}
	\end{equation}
	where, $\Omega = \dfrac{ms[(n-nm)s^{2} + (nm+1)b^{2}]}{[(1-m)s^{2}+b^{2}m]^{2}}, $ also $v\in \mathfrak{m}$ corresponds to 1-form $\beta$ and $\mathfrak{m}$ is identified with the tangent space $T_{H}(G/H)$ of $G/H$ at the origin $H.$
	We denote $\Omega_{i}=\dfrac{\partial \Omega}{\partial y^{i}},$   $\Omega_{j} = \dfrac{\partial \Omega}{\partial y^{j}}$ and 
		$ \Omega_{ij} = \dfrac{\partial^{2} \Omega }{\partial y^{i}y^{j}}.$
\end{theorem}

\begin{proof}

	From Theorem \ref{Thm1}, recall that S-curvature at origin given by equation \ref{eq2} is as follows: 
$$	\dfrac{ms[(n-nm)s^{2}+(nm+1)b^{2}]}{[(1-m)s^{2}+b^{2}m]^{2}}\Big[ \dfrac{1}{\alpha} \big \langle [v,y]_{m},y\big \rangle-\dfrac{m}{(m+1)s} \big \langle [v,y]_{m},v \big \rangle  \Big]. $$\\ \\
Let $\Omega = \dfrac{ms[(n-nm)s^{2} + (nm+1)b^{2}]}{[(1-m)s^{2}+b^{2}m]^{2}}, $ which gives \\ \\
	$S(H,y) = \dfrac{\Omega}{\alpha}  \big \langle [v,y]_{m},y\big \rangle -  \dfrac{m\Omega}{(m+1)s} \big \langle [v,y]_{m},v   \big \rangle $\\ \\
	Consider $S(H,y) = A + B,$ where $A = \dfrac{\Omega}{\alpha}  \big \langle [v,y]_{m},y\big \rangle$ and $B= \dfrac{-m\Omega}{(m+1)s} \big \langle [v,y]_{m},v   \big \rangle. $\\ 
	Now, by definition of mean Berwald curvature 
	\begin{equation}{\label{E}}
    E_{ij} = \dfrac{1}{2}\dfrac{\partial^{2}S(H,y)}{\partial y^{i} \partial y^{j}} = \dfrac{1}{2} \Big( \dfrac{\partial^{2}A}{\partial y^{i} y^{j}} + \dfrac{\partial^{2}B}{\partial y^{i} y^{j}}\Big).
    \end{equation}
	
	Next, we calculate $\dfrac{\partial^{2}A}{\partial y^{i} y^{j}}$ and $\dfrac{\partial^{2}B}{\partial y^{i} y^{j}}$ as follows:\\
	
	$\dfrac{\partial A}{\partial y^{j}} = \dfrac{\partial}{\partial y^{j}}\left( \dfrac{\Omega}{\alpha}  \big \langle [v,y]_{m},y\big \rangle\right) $
	$ = \left( \dfrac{1}{\alpha} \dfrac{\partial \Omega}{\partial y^{j}} -\dfrac{\Omega}{\alpha^{2}} \dfrac{y^{j}}{\alpha} \right) \big \langle [v,y]_{m},y \rangle + \dfrac{\Omega}{\alpha} \left( \big \langle [v,v_{j}]_{m} \big \rangle +\big \langle [v,y]_{m},v_{j} \big \rangle   \right) $
	
	$\dfrac{\partial^{2}A}{\partial y^{i} y^{j}} = \dfrac{\partial}{\partial y^{i}} \left( \dfrac{1}{\alpha} \dfrac{\partial \Omega}{\partial y^{j}} -\dfrac{\Omega}{\alpha^{2}} \dfrac{y^{j}}{\alpha} \right) \big \langle [v,y]_{m},y \rangle + \dfrac{\Omega}{\alpha} \Big( \big \langle [v,v_{j}]_{m} \big \rangle +\big \langle [v,y]_{m},v_{j} \big \rangle   \Big) $
	\begin{equation}{\label{eqq1}}
	\begin{split}
    = & \Big(\frac{1}{\alpha}\Omega_{ij}- \dfrac{y_{i}}{\alpha^{3}} \Omega_{j} -\dfrac{y^{j}}{\alpha^{3}}\Omega_{i}-\dfrac{A}{\alpha^{3}}\delta^{j}_{i}+\dfrac{3\Omega}{\alpha^{5}}y^{i}y^{j}  \Big) \big \langle [v,y]_{m},y \big \rangle \\
	& + \Big( \dfrac{1}{\alpha}\Omega_{j} - \dfrac{\Omega y_{j}}{\alpha^{3}}  \Big) \Big( \big \langle [v,v_{i}]_{m} ,y \big \rangle + \big \langle [v,y]_{m}, v_{i} \big \rangle \Big) \\
	& +\Big( \dfrac{1}{\alpha}\Omega_{i} - \dfrac{\Omega y_{i}}{\alpha^{3}}  \Big) \Big(  \big \langle [v,v_{j}]_{m} ,y \big \rangle + \big \langle [v,y]_{m}, v_{j} \big \rangle  \Big) \\
	& + \dfrac{\Omega}{\alpha}\Big( \  \big \langle [v,v_{j}]_{m} ,v{i} \big \rangle + \big \langle [v,v_{i}]_{m}, v_{j} \big \rangle   \Big)
	\end{split}
	\end{equation}
	Now we wish to calculate $\dfrac{\partial^{2}B}{\partial y^{i} y^{j}}$,\\ \\ 
	Consider 
	\begin{equation}
	\begin{split}
	\dfrac{\partial B}{\partial y^{j}} = & \dfrac{\partial}{\partial y^{j}} \Big[\dfrac{-m\Omega}{(m+1)s} \big \langle [v,y]_{m},v   \big \rangle \Big] \\
	&  = \Big[  \dfrac{-m\Omega}{(m+1)s^{2}}   s_{y^{j}} + \dfrac{m}{(m+1)s} \Omega_{j} \Big] \big \langle [v,y]_{m},v \big \rangle +  \Big[ \dfrac{m \Omega}{(m+1)s} \Big] \big \langle [v,v_{j}]_{m}, v \big \rangle  
	\end{split}
	\end{equation}
\begin{equation}{\label{eqq2}}
\begin{split}
 \dfrac{\partial^{2}B}{\partial y^{i} y^{j}} = & \dfrac{\partial}{\partial y^{i}} \Bigg[ \Big(   \dfrac{-m\Omega}{(m+1)s^{2}}   s_{y^{j}} + \dfrac{m}{(m+1)s} \Omega_{j} \Big)   \big \langle [v,y]_{m},v \big \rangle +  \Big( \dfrac{m \Omega}{(m+1)s} \Big) \big \langle [v,v_{j}]_{m}, v \big \rangle  \Bigg]\\
  & = \Big[ \dfrac{-m}{m+1} \dfrac{\Omega_{i}}{s^{2}}s_{y^{j}} 
 + \dfrac{2m}{(m+1)s^{3}} \Omega s_{y^{i}} s_{y^{j}}  - \dfrac{m}{(m+1)s^{2}} \Omega s_{y^{j}y^{i}} + \dfrac{m}{(m+1)s} \Omega_{ji} \\
 &  - \dfrac{m}{(m+1)s^{2}}s^{y^{i}}\Omega_{j} \Big]  \big \langle [v,y]_{m},v  \big \rangle + \Big[ \dfrac{-m\Omega}{(m+1)s^{2}}s_{y^{j}}+ \dfrac{m}{(m+1)s} \Omega_{j} \Big] \big \langle [v,v_{i}]_{m},v  \big \rangle.
 \end{split}
	\end{equation}

	Using equation \ref{eqq1} and equation \ref{eqq2} in equation \ref{E}, we get 
	
		\begin{equation}
	\begin{split}
	E_{i,j}(H,y) = & \dfrac{1}{2}\Bigg[ \Big(\frac{1}{\alpha}\Omega_{ij}- \dfrac{y_{i}}{\alpha^{3}} \Omega_{j} -\dfrac{y^{j}}{\alpha^{3}}\Omega_{i}-\dfrac{A}{\alpha^{3}}\delta^{j}_{i}+\dfrac{3\Omega}{\alpha^{5}}y^{i}y^{j}  \Big) \big \langle [v,y]_{m},y \big \rangle \\
	& + \Big( \dfrac{1}{\alpha}\Omega_{j} - \dfrac{\Omega y_{j}}{\alpha^{3}}  \Big) \Big( \big \langle [v,v_{i}]_{m} ,y \big \rangle + \big \langle [v,y]_{m}, v_{i} \big \rangle \Big) \\
	& +\Big( \dfrac{1}{\alpha}\Omega_{i} - \dfrac{\Omega y_{i}}{\alpha^{3}}  \Big) \Big(  \big \langle [v,v_{j}]_{m} ,y \big \rangle + \big \langle [v,y]_{m}, v_{j} \big \rangle  \Big) \\
	& + \dfrac{\Omega}{\alpha}\Big( \  \big \langle [v,v_{j}]_{m} ,v{i} \big \rangle + \big \langle [v,v_{i}]_{m}, v_{j} \big \rangle   \Big) 
	+ \Big( \dfrac{-m}{m+1} \dfrac{\Omega_{i}}{s^{2}}s_{y^{j}} 
	+ \dfrac{2m}{(m+1)s^{3}} \Omega s_{y^{i}} s_{y^{j}} \\
	& - \dfrac{m}{(m+1)s^{2}} \Omega s_{y^{j}y^{i}} + \dfrac{m}{(m+1)s} \Omega_{ji} - \dfrac{m}{(m+1)s^{2}}s^{y^{i}}\Omega_{j} \Big)   \big \langle [v,y]_{m},v  \big \rangle  \\
	& + \Big( \dfrac{-m\Omega}{(m+1)s^{2}}s_{y^{j}}+ \dfrac{m}{(m+1)s} \Omega_{j} \Big) \big \langle [v,v_{i}]_{m},v  \big \rangle 
	\Bigg].
	\end{split}
	\end{equation}

\end{proof}

	\section*{Acknowledgments}
 Second author  is very much thankful to UGC for providing financial assistance in terms of JRF fellowship vide letter with UGC-Ref no.:1010/(CSIR-UGC NET DEC 2018). Third author is thankful to UGC for providing financial assistace in terms of JRF fellowship vide letter with UGC-Ref no. :961/(CSIR-UGC NET DEC 2018).


\begin{thebibliography}{30}
	
	
	\bibitem{An} H. An, S. Deng, Invariant $(\alpha,\beta)$-metrics on homogeneous manifolds, \textit{Monatsh Math}, \textbf{154} (2008), 89–102.
	
	\bibitem{BCS} D. Bao, S. S. Chern, Z. Shen, \textit{An Introduction to Riemann-Finsler Geometry,} GTM-
	200, Springer-Verlag 2000.
 
 \bibitem{ChengShen} X. Cheng and Z. Shen, A class of Finsler metrics with isotropic S-curvature, \textit{Israel Journal of Mathematics,} \textbf{169} (2009), 317-340.
\bibitem{Chern2004} S. S. Chern, Z. Shen, \textit{Riemann-Finsler Geometry,} Nankai Tracts in Mathematics, Vol. 6, World Scientifc Publishers, 2005.
 \bibitem {Clayton2015} J. D. Clayton, On Finsler Geometry and Applications in Mechanics: Review and New Perspectives, \textit{Advances in Mathematical Physics,} {\bf 2015} (2015),  828475 (11 pages).

\bibitem{Deng2009}S. Deng, The S-curvature of homogeneous Randers spaces, \textit{Differ.
	Geom. Appl.,} \textbf{27} (2009), 75-84.
\bibitem{Deng2012}	S. Deng, \textit{Homogeneous Finsler Spaces,} Springer Monographs in Mathematics, New York,
2012.
\bibitem{Deng2010} S. Deng, X. Wang, The S-curvature of homogeneous $(\alpha,\beta)$-metrics,
\textit{Balkan J. Geom. Appl.,} \textbf{15} (2) (2010), 39-48.

	\bibitem {Kroner} E. Kr$\ddot{o}$ner, Interrelations between various branches of continuum mechanics,\textit{ Mechanics of Generalized Continua,} (1968), 330-340.
	\bibitem{Latifi} D. Latifi, A. Razavi, On homogeneous Finsler spaces, \textit{Rep. Math. Phys.,} \textbf{57} (2006), 357-366.

	\bibitem{Milnor} J. Milnor, Curvature of left invariant Riemannian metrics on Lie groups, \textit{Adv. Math.,} \textbf{21} (1976), 293-329.
	
		\bibitem{sarita2} S. Rani, G. Shanker, on S-curvature of homogeneous Finsler spaces with Randers changed square metric, \textit{Facta Universitatis Series Mathematics and Informatics} {\bf 35,} (3) (2020), 673-691.



\bibitem{Kiran} G. Shanker and K. Kaur, Homogeneous Finsler space with infinite series $(\alpha, \beta)$-metric, \textit{Applied Sciences}, \textbf{21} (2019), 219-235.
	\bibitem{Sarita} G. Shanker, S. Rani, On S-curvature of a homogeneous Finsler space with square metric,  {\em International Journal of Geometric Methods in Modern Physics,} \textbf{17} (02) (2020), 2050019 (16 Pages).

	\bibitem{Shen1997} Z. Shen, Volume comparison and its applications in Riemann-Finsler geometry, \textit{Adv. Math.,}\textbf{128} (1997), 306-328.
	\bibitem{Shen}Z. Shen, On some non-Riemannian quantities in Finsler geometry, \textit{Canad. Math. Bull.,} \textbf{56} (2013), 184-193.

	\bibitem{Tayebi} A. Tayebi, B. Najafi, Weakly stretch Finsler metrics, \textit{Publ. Math. Debercen,} \textbf{91} (2017), 441-454.

\end{thebibliography}
\end{document}